\documentclass[a4paper,12pt]{amsart}

\usepackage{amsfonts}
\usepackage{amssymb}
\usepackage{amsthm,color}

\allowdisplaybreaks

\newtheorem{thm}{Theorem}[section]

\newtheorem{lem}[thm]{Lemma}
\newtheorem{cor}[thm]{Corollary}

\theoremstyle{definition} 

\newtheorem{remark}[thm]{Remark}

\numberwithin{equation}{section}

\def\supp{{\mathop\mathrm{\,supp\,}}}

\def\laz{\langle}
\def\raz{\rangle}

\allowdisplaybreaks
\begin{document}

\title{A note on extrapolation of compactness}

\author{Shenyu Liu, Huoxiong Wu and Dongyong Yang}

\address{Shenyu Liu, School of Mathematical Sciences\\
 Xiamen University\\
 Xiamen 361005, China}

\email{shenyuliu@stu.xmu.edu.cn}

\address{Huoxiong Wu, School of Mathematical Sciences\\
 Xiamen University\\
 Xiamen 361005, China}

\email{huoxwu@xmu.edu.cn}

\address{Dongyong Yang (Corresponding author), School of Mathematical Sciences\\
 Xiamen University\\
 Xiamen 361005, China}

\email{dyyang@xmu.edu.cn}

\makeatletter
\@namedef{subjclassname@2020}{\textup{2020} Mathematics Subject Classification}
\makeatother

\date{\today}



\begin{abstract}
This note is devoted to the study of Hyt\"{o}nen's extrapolation theorem of compactness on weighted Lebesgue spaces. Two criteria of compactness of linear operators in the two-weight setting are obtained. As applications, we obtain two-weight compactness of commutators of Calder\'{o}n--Zygmund operators, fractional integrals and bilinear Calder\'{o}n--Zygmund operators.
\end{abstract}

\maketitle


\section{Introduction}\label{s-1}
By a weight on $\mathbb{R}^d$ we mean a locally integrable function $w\in L^1_{\rm loc}(\mathbb{R}^d)$ that is positive almost everywhere. For $p\in (1,\infty)$, we define the weighted Lebesgue spaces $$L^p(w):=\bigg\{f:\mathbb{R}^d\to \mathbb{C}\ {\rm measurable}\ \bigg|\  \|f\|_{L^p(w)}:=\bigg(\int_{\mathbb{R}^d}|f(x)|^p w(x)dx\bigg)^{\frac{1}{p}}<\infty\bigg\}.$$
A weight $w$ is called a Muckenhoupt $A_p$ weight (write $w\in A_p$) if $$[w]_{A_p}:=\sup_{Q\subset \mathbb{R}^d}\laz w\raz_Q\laz w^{1-p'}\raz^{p-1}_Q< \infty,$$ where the supremum is taken over all cubes $Q\subset \mathbb{R}^d$, $\laz w\raz_Q:=|Q|^{-1}\int_Qw$ and $p':=\frac{p}{p-1}$. We define $A_{\infty}:=\cup_{p>1}A_p$ and $$[w]_{A_{\infty}}:=\sup_{Q\subset \mathbb{R}^d}\frac{1}{w(Q)}\int_QM(w\chi_Q)(x)dx<\infty,$$ where $M$ denotes the Hardy--Littlewood maximal operator. A weight $w$ is called an $A_{p,q}$ weight (write $w\in A_{p,q}$) if $$[w]_{A_{p,q}}:=\sup_{Q\subset \mathbb{R}^d}\laz w^q\raz^{\frac{1}{q}}_Q\laz w^{-p'}\raz^{\frac{1}{p'}}_Q< \infty,\ \ 1<p\leq q<\infty.$$ It readily follows that $$w\in A_{p,q}\iff w^q\in A_{1+q/p'}\subset A_q\iff w^{-p'}\in A_{1+p'/q}\subset A_{p'}\Rightarrow w^p\in A_p.$$

In this note, we show that the corresponding results of Hyt\"{o}nen \cite{H}, Hyt\"{o}nen and Lappas \cite{HL} are still valid in the two-weight setting. To be precise, we establish the following two extrapolation theorems of compactness:
\begin{thm}\label{t-main}
Let $T$ be a linear operator simultaneously defined and bounded from $L^{p}(\sigma)$ to $L^{p}(\lambda)$ for all $p\in (1,\infty)$ and all $\sigma, \lambda\in A_p$. Suppose moreover that $T$ is compact on $L^q(w)$ for some $q\in (1,\infty)$ and some $w\in A_q$. Then $T$ is compact from $L^{p}(\sigma)$ to $L^{p}(\lambda)$ for all $p\in (1,\infty)$ and all $\sigma, \lambda\in A_p$.
\end{thm}

Theorem \ref{t-main} says that one-weight compactness (or even unweighted compactness when we take $w\equiv 1$) bootstraps to two-weight compactness, if the two-weight boundedness is already known. We also consider the counterpart off-diagonal case in the two-weight setting as follows:
\begin{thm}\label{t-main2}
Let $T$ be a linear operator simultaneously defined and compact from $L^{p_1}(w_1^{p_1})$ to $L^{q_1}(w_1^{q_1})$ for some $1<p_1\leq q_1<\infty$ and some $w_1\in A_{p_1,q_1}$. Suppose moreover that $T$ is simultaneously defined and bounded from $L^{p}(\sigma^p)$ to $L^{q}(\lambda^q)$ for all $p,q\in (1,\infty)$ with $\frac{1}{p}-\frac{1}{q}=\frac{1}{p_1}-\frac{1}{q_1}$ and all $\sigma,\lambda \in A_{p,q}$. Then $T$ is compact from $L^{p}(\sigma^p)$ to $L^{q}(\lambda^q)$ for all $p,q\in (1,\infty)$ with $\frac{1}{p}-\frac{1}{q}=\frac{1}{p_1}-\frac{1}{q_1}$ and all $\sigma,\lambda \in A_{p,q}$.
\end{thm}

The rest of this note is organized as follows. In Section \ref{s-2}, we recover Hyt\"{o}nen's extrapolation theorem by extending the extrapolation theorem of the $A_p$ weights in \cite{H}. In Section \ref{s-3}, we obtain some generalizations of extrapolation of compactness in the two-weight setting. Section \ref{s-4} is devoted to providing two applications to the two-weight compactness of linear operators, including commutators of Calder\'{o}n--Zygmund operators and fractional integrals. We also give an application to the two-weight compactness of commutators of bilinear Calder\'{o}n--Zygmund operators in Section \ref{s-5}.

Throughout this note, we write $A\lesssim B$ to denote $A \leq cB$ for some positive constant $c$ which is independent of the main parameters. We write $A\approx B$ if $A\lesssim B$ and $B\lesssim A$.

\section{Hyt\"{o}nen's extrapolation theorem: a revisit}\label{s-2}

Hyt\"{o}nen \cite{H} first established a compact version of Rubio de Francia's weighted extrapolation theorem (see Theorem \ref{t-h} below), whose approach relies on interpolation of compactness, see also \cite{HL,HLb,COY}. In this section, we recover Hyt\"{o}nen's extrapolation theorem by applying an abstract tool about extrapolation of compact operators due to Cwikel \cite[Theorem 2.1]{C}, in which there is no additional assumption on Banach couples in comparison to \cite[Theorem 3.1]{H} (see also Theorem \ref{t-ck} below).

\begin{thm}[\cite{C}]\label{t-c}
Let $(X_{0}, X_{1})$ and $(Y_{0}, Y_{1})$ be Banach couples and $T$ be a linear operator such that $T: X_{0}+X_{1}\to Y_{0}+Y_{1}$ and $T: X_{j} \to Y_{j}$ boundedly for $j=0,1$. Let $[\ ,\ ]_{\theta}$ be the complex interpolation functor of Calder\'{o}n. Suppose moreover that $T:[X_{0}, X_{1}]_{\theta} \to [Y_{0}, Y_{1}]_{\theta}$ is compact for at least one value $\theta_*$ of $\theta$ in $(0,1)$. Then $T:[X_{0}, X_{1}]_{\theta} \to [Y_{0}, Y_{1}]_{\theta}$ is compact for all values of $\theta$ in $(0,1)$.
\end{thm}

To apply Theorem \ref{t-c}, we need a well-known result of Stein and Weiss \cite{SW} as follows, see also \cite[Theorem 5.5.3]{BL}.

\begin{thm}[\cite{SW,BL}]\label{t-sw}
Let $q,t\in [1,\infty)$ and $w_0,w_1$ be two weights. Then for all $\theta \in(0,1)$, $$L^p(w)=[L^t(w_0),L^q(w_1)]_{\theta},$$ where
\begin{equation}
\frac{1}{p}=\frac{1-\theta}{t}+\frac{\theta}{q},\ \ w^{\frac{1}{p}}=w^{\frac{1-\theta}{t}}_0w^{\frac{\theta}{q}}_1.\label{eq-pw}
\end{equation}
\end{thm}

We also need the following powerful result of Hyt\"{o}nen \cite[Lemma 4.3]{H} that can be understood as an extrapolation of the $A_p$ weights.

\begin{thm}[\cite{H}]\label{t-h1}
Let $p,q\in (1,\infty)$, $w\in A_p$ and $w_1\in A_q$. Then there exist $\theta\in (0,1)$, $t\in (1,\infty)$ and $w_0\in A_t$ such that \eqref{eq-pw} holds.
\end{thm}

To adapt Theorem \ref{t-sw} to our case, we need a further result described below:

\begin{lem}\label{l-w}
Let $p,q\in (1,\infty)$, $u\in A_p$ and $v\in A_q$. Then there exist $\theta_0,\theta_1\in (0,1)$, $r,s\in (1,\infty)$, $w_0\in A_r$ and $w_1\in A_s$ such that
\begin{align}
\frac{1}{p}&=\frac{1-\theta_0}{r}+\frac{\theta_0}{s},\ \ u^{\frac{1}{p}}=w_0^{\frac{1-\theta_0}{r}}w_1^{\frac{\theta_0}{s}},\label{eq-w1}\\
\frac{1}{q}&=\frac{1-\theta_1}{r}+\frac{\theta_1}{s},\ \ v^{\frac{1}{q}}=w_0^{\frac{1-\theta_1}{r}}w_1^{\frac{\theta_1}{s}}.\label{eq-w2}
\end{align}
\end{lem}

\begin{proof}
For $v\in A_q$ and $u\in A_p$, from Theorem \ref{t-h1}, there exist $\tilde{\theta}_0\in (0,1)$, $r\in (1,\infty)$ and $w_0\in A_r$ such that
\begin{equation}
\frac{1}{q}=\frac{1-\tilde{\theta}_0}{r}+\frac{\tilde{\theta}_0}{p},\ \ v^{\frac{1}{q}}=w_0^{\frac{1-\tilde{\theta}_0}{r}}u^{\frac{\tilde{\theta}_0}{p}}.\label{eq-w3}
\end{equation}
Then for $u\in A_p$ and  $w_0\in A_r$, again by Theorem \ref{t-h1}, there exist $\theta_0\in (0,1)$, $s\in (1,\infty)$ and $w_1\in A_s$ such that \eqref{eq-w1} holds. Combined with \eqref{eq-w1} and \eqref{eq-w3}, we obtain \eqref{eq-w2} by taking $\theta_1=\theta_0\tilde{\theta}_0$ which completes the proof.
\end{proof}

By combining with Theorem \ref{t-sw} and Lemma \ref{l-w}, we immediately obtain the following interpolation result.
\begin{cor}\label{c-w}
Let $\lambda\in [1,\infty)$, $p,q\in (\lambda,\infty)$, $u\in A_{p/\lambda}$ and $v\in A_{q/\lambda}$. Then
\begin{equation*}
L^p(u)=[L^r(w_0),L^s(w_1)]_{\theta_0},\ \ L^q(v)=[L^r(w_0),L^s(w_1)]_{\theta_1}
\end{equation*}
for some $\theta_0,\theta_1\in (0,1)$, $r,s\in (\lambda,\infty)$, $w_0\in A_{r/\lambda}$ and $w_1\in A_{s/\lambda}$.
\end{cor}

Now we are ready to recover Hyt\"{o}nen's extrapolation theorem in \cite[Theorem 1.2]{H}; see also \cite{SVW} for the weighted compactness of Calder\'on-Zygmund operators.

\begin{thm}[\cite{H}]\label{t-h}
Let $\lambda\in [1,\infty)$, $p_1\in (\lambda,\infty)$ and $T$ be a linear operator simultaneously defined and bounded on $L^{p_1}(\tilde{w})$ for all $\tilde{w}\in A_{p_1/\lambda}$, with the operator norm dominated by some increasing function of $[\tilde{w}]_{A_{p_1/\lambda}}$. Suppose in addition that $T$ is compact on $L^{p_1}(w_1)$ for some $w_1\in A_{p_1/\lambda}$. Then $T$ is compact on $L^p(w)$ for all $p\in (\lambda,\infty)$ and all $w\in A_{p/\lambda}$.
\end{thm}

\begin{proof}
First, by a rescaling version of the Rubio de Francia weighted extrapolation theorem (see, for example, \cite[Theorem 1.1]{H}), we have that $T$ is a bounded linear operator on $L^p(w)$ for all $p\in (\lambda,\infty)$ and all $w\in A_{p/\lambda}$. Next, for any fixed $p\in (\lambda,\infty)$ and $w\in A_{p/\lambda}$, since $T$ is compact on $L^{p_1}(w_1)$ for $p_1\in (\lambda,\infty)$ and some $w_1\in A_{p_1/\lambda}$, by Corollary \ref{c-w},
\begin{equation*}
L^p(w)=[L^r(w_0),L^s(w_2)]_{\theta_0},\ \ L^{p_1}(w_1)=[L^r(w_0),L^s(w_2)]_{\theta_1}
\end{equation*}
for some $\theta_0,\theta_1\in (0,1)$, $r,s\in (\lambda,\infty)$, $w_0\in A_{r/\lambda}$ and $w_2\in A_{s/\lambda}$. Finally, by Theorem \ref{t-c} with $X_0=Y_0=L^r(w_0)$ and $X_1=Y_1=L^s(w_2)$, we obtain that $T$ is compact on $L^p(w)$. This completes the proof.
\end{proof}

Note that Theorem \ref{t-h} only deals with linear operators and the classical Muckenhoupt $A_p$ weight class. It would be interesting to extend it to the case of sublinear operators (e.g., see \cite{WX}), or more general weight classes (e.g., see \cite{SVW}).

\section{Extensions in the two-weight setting}\label{s-3}
In this section, we present the proofs of Theorems \ref{t-main} and \ref{t-main2}. Our abstract tool is the following complex interpolation theorem of compactness due to Cwikel and Kalton \cite{CK}.

\begin{thm}[\cite{CK}]\label{t-ck}
Let $(X_{0}, X_{1})$ and $(Y_{0}, Y_{1})$ be Banach couples and $T$ be a linear operator such that $T: X_{0}+X_{1}\to Y_{0}+Y_{1}$ and $T: X_{j} \to Y_{j}$ boundedly for $j=0,1$. Suppose moreover that $T: X_{1} \to Y_{1}$ is compact. Then $T:[X_{0}, X_{1}]_{\theta} \to [Y_{0}, Y_{1}]_{\theta}$ is compact for $\theta \in(0,1)$ under any of the following four side conditions:
\begin{enumerate}
\item[$(a)$] $X_{1}$ has the UMD (unconditional martingale differences) property;
\item[$(b)$] $X_{1}$ is reflexive, and $X_{1}=[X_{0}, E]_{\alpha}$ for some Banach space $E$ and $\alpha \in (0,1)$;
\item[$(c)$] $Y_{1}=[Y_{0}, F]_{\beta}$ for some Banach space $F$ and $\beta \in(0,1)$;
\item[$(d)$] $X_{0}$ and $X_{1}$ are both complexified Banach lattices of measurable functions on a common measure space.
\end{enumerate}
\end{thm}

We have swapped the roles of the indices $0$ and $1$ in comparison to \cite{CK}. It is known that all four side conditions hold for the weighted Lebesgue spaces, see \cite[Corollary 3.3]{H}. To obtain our extrapolation theorems of two-weight compactness, we need the following key result which is essentially an extension of the reverse H\"{o}lder inequality.
\begin{lem}\label{l-tw}
Let $p,q\in (1,\infty)$, $\sigma,\lambda\in A_p$ and $w\in A_q$. Then there exist $\theta\in (0,1)$, $t\in (1,\infty)$ and $\sigma_0,\lambda_0\in A_t$ such that
\begin{equation}
\frac{1}{p}=\frac{1-\theta}{t}+\frac{\theta}{q},\ \ \sigma^{\frac{1}{p}}=\sigma^{\frac{1-\theta}{t}}_0w^{\frac{\theta}{q}},\ \ \lambda^{\frac{1}{p}}=\lambda^{\frac{1-\theta}{t}}_0w^{\frac{\theta}{q}}.\label{eq-ptw}
\end{equation}
\end{lem}

We recall that a weight $w$ satisfies the reverse H\"{o}lder inequality for some $\eta\in (1,\infty)$ (write $w\in RH_{\eta}$) if $$[w]_{RH_{\eta}}:=\sup_{Q\subset \mathbb{R}^d}\laz w^{\eta}\raz^{\frac{1}{\eta}}_Q \laz w\raz^{-1}_Q <\infty.$$Notice that the reverse H\"{o}lder inequality can simultaneously hold for a finite number of the Muckenhoupt weights. From this fact, one may see that the proof of Lemma \ref{l-tw} below is similar to that in \cite[Lemma 4.3]{H} with some minor changes.

\begin{proof}[Proof of Lemma \ref{l-tw}]
By a simple calculation of \eqref{eq-ptw}, we have that
\begin{equation*}
t=t(\theta)=\frac{1-\theta}{\frac{1}{p}-\frac{\theta}{q}},\ \ \sigma_0=\sigma_0(\theta)=\sigma^{\frac{t}{p(1-\theta)}}w^{-\frac{t\theta}{q(1-\theta)}},\ \ \lambda_0=\lambda_0(\theta)=\lambda^{\frac{t}{p(1-\theta)}}w^{-\frac{t\theta}{q(1-\theta)}},
\end{equation*}
from which it suffices to show that there exists $\theta\in (0,1)$ such that $t\in (1,\infty)$ and $\sigma_0,\lambda_0\in A_t$. Note that $t(0)=p\in (1,\infty)$. Then by continuity, we have $t\in (1,\infty)$ for small enough $\theta>0$.

We now show that $\sigma_0,\lambda_0\in A_t$. Observe that the weights $\sigma$ and $\lambda$ are independent of each other, so we deal with these two weights separately in the same way as \cite[Lemma 4.3]{H}. Let
\begin{equation*}
\epsilon:= \frac{\theta p}{q'},\ \ \delta:=\frac{\theta p'}{q},\ \ r(\theta):=\frac{t(1+\epsilon)}{p(1-\theta)}=\frac{t(q'+\theta p)}{pq'(1-\theta)},\ \ s(\theta):=\frac{t'(1+\delta)}{p'(1-\theta)}=\frac{t'(q+\theta p')}{p'q(1-\theta)}.
\end{equation*}
A direct calculation yields that for any cube $Q\subset \mathbb{R}^d$,
\begin{align}\label{eq-twa}
\laz \sigma_0\raz_Q \laz \sigma_0^{1-t'}\raz_Q^{t-1}&= \laz \sigma^{\frac{t}{p(1-\theta)}}w^{-\frac{t\theta}{q(1-\theta)}} \raz_Q \laz \sigma^{-\frac{t'}{p(1-\theta)}}w^{\frac{t'\theta}{q(1-\theta)}} \raz_Q^{t-1} \\
&\leq \laz \sigma^{\frac{t(1+\epsilon)}{p(1-\theta)}} \raz^{\frac{1}{1+\epsilon}}_Q \laz w^{-\frac{t\theta}{q(1-\theta)} \frac{q'+\theta p}{\theta p}} \raz^{\frac{\epsilon}{1+\epsilon}}_Q \laz \sigma^{-\frac{t'(1+\delta)}{p(1-\theta)}}\raz_Q^{\frac{t-1}{1+\delta}} \laz w^{\frac{t'\theta}{q(1-\theta)} \frac{q+\theta p'}{\theta p'}} \raz_Q^{\frac{\delta (t-1)}{1+\delta}} \notag\\
&= \laz \sigma^{r(\theta)} \raz^{\frac{1}{1+\epsilon}}_Q \laz w^{(1-q')\frac{t(q'+\theta p)}{pq'(1-\theta)}} \raz^{\frac{\epsilon}{1+\epsilon}}_Q \laz \sigma^{(1-p')\frac{t'(1+\delta)}{p'(1-\theta)}}\raz_Q^{\frac{t-1}{1+\delta}} \laz w^{\frac{t'(q+\theta p')}{p'q(1-\theta)}} \raz_Q^{\frac{\delta (t-1)}{1+\delta}} \notag\\
&= \laz \sigma^{r(\theta)}\raz_Q^{\frac{1}{1+\epsilon}} \laz (w^{1-q'})^{r(\theta)}\raz_Q^{\frac{\epsilon}{1+\epsilon}} \laz(\sigma^{1-p'})^{s(\theta)}\raz_Q^{\frac{t-1}{1+\delta}} \laz w^{s(\theta)}\raz_Q^{\frac{\delta(t-1)}{1+\delta}},\notag
\end{align}
and similarly,
\begin{align}\label{eq-twb}
\laz \lambda_0\raz_Q \laz \lambda_0^{1-t'}\raz_Q^{t-1}&\leq \laz \lambda^{r(\theta)}\raz_Q^{\frac{1}{1+\epsilon}} \laz (w^{1-q'})^{r(\theta)}\raz_Q^{\frac{\epsilon}{1+\epsilon}} \laz(\lambda^{1-p'})^{s(\theta)}\raz_Q^{\frac{t-1}{1+\delta}} \laz w^{s(\theta)}\raz_Q^{\frac{\delta(t-1)}{1+\delta}}.
\end{align}

Observe that $r(0)=s(0)=1$. Another application of continuity yields that for any fixed $\eta>1$, $\max(r(\theta),s(\theta))\leq \eta$ for small enough $\theta>0$. Moreover, since $w\in A_q$, $w^{1-q'}\in A_{q'}$, $\sigma,\lambda \in A_p$ and $\sigma^{1-p'},\lambda^{1-p'}\in A_{p'}$, then there exists $\eta>1$ such that each of the above weights satisfies the reverse H\"{o}lder inequality. Thus, for small enough $\theta>0$, using the reverse H\"{o}lder inequality, it follows from \eqref{eq-twa} and \eqref{eq-twb} that
\begin{align*}
\laz \sigma_0\raz_Q \laz \sigma_0^{1-t'}\raz_Q^{t-1} &\lesssim \laz \sigma\raz_Q^{r(\theta)\frac{1}{1+\epsilon}} \laz w^{1-q'}\raz_Q^{r(\theta)\frac{\epsilon}{1+\epsilon}} \laz\sigma^{1-p'}\raz_Q^{s(\theta)\frac{t-1}{1+\delta}} \laz w\raz_Q^{s(\theta)\frac{\delta(t-1)}{1+\delta}},\\
\laz \lambda_0\raz_Q \laz \lambda_0^{1-t'}\raz_Q^{t-1}&\lesssim \laz \lambda\raz_Q^{r(\theta)\frac{1}{1+\epsilon}} \laz w^{1-q'}\raz_Q^{r(\theta)\frac{\epsilon}{1+\epsilon}} \laz\lambda^{1-p'}\raz_Q^{s(\theta)\frac{t-1}{1+\delta}} \laz w\raz_Q^{s(\theta)\frac{\delta(t-1)}{1+\delta}}.
\end{align*}
Observe that
\begin{align*}
\laz \sigma\raz_Q^{r(\theta)\frac{1}{1+\epsilon}} \laz\sigma^{1-p'}\raz_Q^{s(\theta)\frac{t-1}{1+\delta}}&= \laz \sigma\raz_Q^{\frac{t}{p(1-\theta)}} \laz\sigma^{1-p'}\raz_Q^{(p-1)\frac{t}{p(1-\theta)}}\leq [\sigma]^{\frac{t}{p(1-\theta)}}_{A_p}=[\sigma]^{\frac{q}{q-p\theta}}_{A_p},\\
\laz w^{1-q'}\raz_Q^{r(\theta)\frac{\epsilon}{1+\epsilon}} \laz w\raz_Q^{s(\theta)\frac{\delta(t-1)}{1+\delta}}&= \laz w^{1-q'}\raz_Q^{(q-1)\frac{t \theta}{q(1-\theta)}} \laz w\raz_Q^{\frac{t \theta}{q(1-\theta)}}\leq [w]^{\frac{t \theta}{q(1-\theta)}}_{A_q}=[w]^{\frac{\theta p}{q-\theta p}}_{A_q}.
\end{align*}
Then we have
\begin{align*}
[\sigma_0]_{A_t} \lesssim [\sigma]^{\frac{q}{q-\theta p}}_{A_p}[w]^{\frac{\theta p}{q-\theta p}}_{A_q}<\infty.
\end{align*}
Similarly,
\begin{align*}
[\lambda_0]_{A_t} \lesssim [\lambda]^{\frac{q}{q-\theta p}}_{A_p}[w]^{\frac{\theta p}{q-\theta p}}_{A_q}<\infty.
\end{align*}
This completes the proof.
\end{proof}

By combining with Theorem \ref{t-sw} and Lemma \ref{l-tw}, we immediately obtain the following interpolation result in the two-weight setting.

\begin{cor}\label{c-tw}
Let $p,q\in (1,\infty)$, $\sigma,\lambda\in A_p$ and $w\in A_q$. Then
\begin{equation}
L^p(\sigma)=[L^t(\sigma_0),L^q(w)]_{\theta},\ \ L^p(\lambda)=[L^t(\lambda_0),L^q(w)]_{\theta} \label{eq-tw}
\end{equation}
for some $\theta\in (0,1)$, $t\in (1,\infty)$ and $\sigma_0,\lambda_0\in A_t$.
\end{cor}

By Theorem \ref{t-sw} and using the same method as Lemma \ref{l-tw} (see also \cite[Lemma 3.4]{HL}) with some minor changes, we further conclude the following interpolation result. The details are omitted here.
\begin{lem}\label{l-tw2}
Let $1<p\leq q<\infty$, $1<p_1\leq q_1<\infty$, $\sigma,\lambda\in A_{p,q}$ and $w_1\in A_{p_1,q_1}$. Then there exist $\theta\in (0,1)$, $1<p_0\leq q_0<\infty$ and $\sigma_0,\lambda_0\in A_{p_0,q_0}$ such that
\begin{equation}
L^p(\sigma^p)=[L^{p_0}(\sigma^{p_0}_0),L^{p_1}(w^{p_1}_1)]_{\theta},\ \ L^q(\lambda^q)=[L^{q_0}(\lambda^{q_0}_0),L^{q_1}(w^{q_1}_1)]_{\theta}, \label{eq-tw2}
\end{equation}
where
\begin{equation*}
\frac{1}{p}=\frac{1-\theta}{p_0}+\frac{\theta}{p_1},\ \ \frac{1}{q}=\frac{1-\theta}{q_0}+\frac{\theta}{q_1},\ \ \sigma=\sigma_0^{1-\theta}w_1^{\theta},\ \ \lambda=\lambda_0^{1-\theta}w_1^{\theta}.
\end{equation*}
\end{lem}
We now turn to verify our main results.
\begin{proof}[Proof of Theorem \ref{t-main}]
For any fixed $p\in (1,\infty)$ and $\sigma,\lambda\in A_p$, it suffices to show that $T$ is compact from $L^{p}(\sigma)$ to $L^{p}(\lambda)$. Recall that $T$ is compact on $L^q(w)$ for some $q\in (1,\infty)$ and some $w\in A_q$. By Corollary \ref{c-tw}, we have that \eqref{eq-tw} holds for some $\theta\in (0,1)$, $t\in (1,\infty)$ and $\sigma_0,\lambda_0\in A_t$. Since $T$ is bounded from $L^t(\sigma_0)$ to $L^t(\lambda_0)$, by Theorem \ref{t-ck}, we obtain that $T$ is compact from $L^{p}(\sigma)$ to $L^{p}(\lambda)$ by taking $X_0=L^t(\sigma_0)$, $Y_0=L^t(\lambda_0)$ and $X_1=Y_1=L^q(w)$. This completes the proof.
\end{proof}

\begin{proof}[Proof of Theorem \ref{t-main2}]
For any fixed $p,q\in (1,\infty)$ with $\frac{1}{p}-\frac{1}{q}=\frac{1}{p_1}-\frac{1}{q_1}$ and $\sigma,\lambda\in A_{p,q}$, it suffices to show that $T$ is compact from $L^{p}(\sigma^p)$ to $L^{q}(\lambda^q)$. Recall that for $1<p_1\leq q_1<\infty$ and $w_1\in A_{p_1,q_1}$, $T$ is compact from $L^{p_1}(w_1^{p_1})$ to $L^{q_1}(w_1^{q_1})$. By Lemma \ref{l-tw2}, we have that \eqref{eq-tw2} holds for some $\theta\in (0,1)$, $p_0,q_0\in (1,\infty)$ with $\frac{1}{p_0}-\frac{1}{q_0}=\frac{1}{p_1}-\frac{1}{q_1}$ and $\sigma_0,\lambda_0\in A_{p_0,q_0}$. Since $T$ is bounded from $L^{p_0}(\sigma^{p_0}_0)$ to $L^{q_0}(\lambda^{q_0}_0)$, by Theorem \ref{t-ck}, we obtain that $T$ is compact from $L^{p}(\sigma^p)$ to $L^{q}(\lambda^q)$ by taking $X_0=L^{p_0}(\sigma^{p_0}_0)$, $Y_0=L^{q_0}(\lambda^{q_0}_0)$, $X_1=L^{p_1}(w^{p_1}_1)$ and $Y_1=L^{q_1}(w^{q_1}_1)$. This completes the proof.
\end{proof}

\section{Applications to the two-weight compactness of linear operators}\label{s-4}

Notice that the unweighted compactness result is probably the most available and relevant case for most applications. Based on this fact, we point out that the nature of extrapolation of two-weight compactness (including the one-weight case) is essentially a combination of the reverse H\"{o}lder inequality and interpolation Theorem \ref{t-ck}. From this point of view, we will give two applications below to the two-weight compactness of linear operators.

Both of our two applications in this section deal with commutators of the form $$[b,T](f):=bT(f)-T(bf),$$ where the pointwise multiplier $b$ is a locally integrable function. We say that a locally integrable function $b$ belongs to ${\rm BMO}_w(\mathbb{R}^d)$ with $w\in A_{\infty}$ if $$\|b\|_{{\rm BMO}_w(\mathbb{R}^d)}:=\sup_{Q\subset \mathbb{R}^d} w(Q)^{-1}\int_Q |b(x)-\laz b\raz_Q|dx <\infty,$$ where the supremum is taken over all cubes $Q\subset \mathbb{R}^d$ and $w(Q):=\int_Q w$. When $w\equiv 1$, we simply write ${\rm BMO}_w(\mathbb{R}^d)$ as ${\rm BMO}(\mathbb{R}^d)$. We then define $${\rm CMO}(\mathbb{R}^d):=\overline{C^{\infty}_c(\mathbb{R}^d)}^{{\rm BMO}(\mathbb{R}^d)},$$ where $C^{\infty}_c(\mathbb{R}^d)$ denotes the space of smooth functions on $\mathbb{R}^d$ with compact support and the closure is in the BMO norm. We say that a function $b\in {\rm BMO}_{w}(\mathbb{R}^d)$ belongs to ${\rm CMO}_{w}(\mathbb{R}^d)$ with $w\in A_{\infty}$ if the following three conditions hold:
\begin{enumerate}
\item[$(a)$] $\lim\limits_{a \to 0^+}\sup\limits_{|Q|=a}w(Q)^{-1}\int_Q |b(x)-\laz b\raz_Q|dx=0;$
\item[$(b)$] $\lim\limits_{a \to \infty}\sup\limits_{|Q|=a}w(Q)^{-1}\int_Q |b(x)-\laz b\raz_Q|dx=0;$
\item[$(c)$] $\lim\limits_{a \to \infty}\sup\limits_{Q\cap Q(0,a)=\emptyset}w(Q)^{-1}\int_Q |b(x)-\laz b\raz_Q|dx=0;$
\end{enumerate}
where $Q(0,a)$ denotes the cube centered at $0$ with the side length $2a$ in $\mathbb{R}^d$. When $w\equiv 1$, it turns to be an equivalent characterization of ${\rm CMO}(\mathbb{R}^d)$ due to Uchiyama \cite[Lemma 3]{U}.
 \vspace{0.3cm}

\noindent {\bf 4.A. Commutators of Calder\'{o}n--Zygmund operators.} In our first application, we study the commutators of Calder\'{o}n--Zygmund operators. Recall that a linear operator $T$ is a Calder\'{o}n--Zygmund operator if it is an integral operator defined initially on $f\in C^{\infty}_c(\mathbb{R}^d)$: $$T(f)(x):=\int_{\mathbb{R}^d}K(x,y)f(y)dy,\ \ x\notin \supp(f),$$ and it extends to a bounded operator on $L^2(\mathbb{R}^d)$, where the kernel $K$ satisfies the standard estimates
\begin{align*}
|K(x,y)|\lesssim \frac{1}{|x-y|^d},\ \ x\not= y,
\end{align*}
and
\begin{align*}
|K(x+h,y)-K(x,y)|+|K(x,y+h)-K(x,y)|\lesssim \frac{|h|^{\delta}}{|x-y|^{d+\delta}}
\end{align*}
for all $|x-y|\geq 2|h|>0$ and some fixed $\delta\in (0,1]$.

The two-weight boundedness of Bloom type for commutators of Calder\'{o}n--Zygmund operators were first considered by Bloom \cite{B} for the Hilbert transform in one dimension. His result was then improved by Holmes, Lacey and Wick \cite{HLW} as follows:
\begin{thm}[\cite{HLW}]\label{t-hlw}
Let $T$ be a Calder\'{o}n--Zygmund operator, $p\in (1,\infty)$, and $\sigma, \lambda \in A_p$. Suppose $b \in {\rm BMO}_{v}(\mathbb{R}^{d})$ with $v={\sigma}^{\frac{1}{p}}\lambda^{-\frac{1}{p}}$. Then
$$\|[b,T]\|_{L^p(\sigma)\to L^p(\lambda)}\lesssim \|b\|_{{\rm BMO}_v(\mathbb{R}^d)}.$$
\end{thm}

For the application of Theorem \ref{t-ck}, we need the following unweighted compactness result of Uchiyama \cite{U} (see also \cite[Theorem 2]{CC}).
\begin{thm}[\cite{U}]\label{t-u}
Let $T$ be a Calder\'{o}n--Zygmund operator and $b\in {\rm CMO}(\mathbb{R}^{d})$. Then the commutator $[b, T]$ is compact on $L^p(\mathbb{R}^d)$ for all $p\in (1,\infty)$.
\end{thm}

The following sharp reverse H\"{o}lder inequality of the $A_p$ weights for $p\in (1,\infty]$ is useful.
\begin{thm}[\cite{P}]\label{t-p}
$(a)$ Let $p\in (1,\infty]$, $w\in A_p$ and
$$r_w=\begin{cases}
1+\frac{1}{2^{2p+d+1}[w]_{A_p}},& p\in(1,\infty);\\
1+\frac{1}{2^{d+11}[w]_{A_{\infty}}},& p=\infty.
\end{cases}
$$
Then for any cube $Q\subset \mathbb{R}^d$, $\laz w^{r_w}\raz^{\frac{1}{r_w}}_Q\leq 2\laz w\raz_Q.$

$(b)$ If a weight $w\in RH_r$ for some $r\in (1,\infty)$, then $[w]_{A_{\infty}}\lesssim [w]_{RH_r} r'$.
\end{thm}
For simplicity, we write $$\tilde{r}_{p,w}:=\min(r_w,r_{w^{1-p'}})=1+\frac{1}{2^{2\max(p,p')+d+1}[w]^{\max(1,p'-1)}_{A_p}}$$ for any $p\in (1,\infty)$ and $w\in A_p$. Note that if $w\in A_{\infty}$, then from Theorem \ref{t-p} $(a)$, $w\in RH_{r_w}$. It follows that for any fixed $\eta\in (1,r_w)$ and any cube $Q\subset \mathbb{R}^d$, $$\laz w^{\eta \frac{r_w}{\eta}}\raz^{\frac{\eta}{r_w}}_Q=\laz w^{r_w}\raz^{\frac{\eta}{r_w}}_Q\leq 2^{\eta} \laz w\raz^{\eta}_Q\leq 2^{\eta} \laz w^{\eta} \raz_Q,$$ which implies $w^{\eta}\in RH_{r_w/\eta}$. Thus, from Theorem \ref{t-p} $(b)$, $[w^{\eta}]_{A_\infty}\leq 2^{\eta}(r_w/\eta)'<\infty$ and hence $w^{\eta}\in A_{\infty}$.

Very recently, Lacey and Li \cite{LL} established the $L^p(\sigma)\to L^p(\lambda)$ compactness for commutators of Calder\'{o}n--Zygmund operators under the assumption $b \in {\rm CMO}_{v}(\mathbb{R}^{d})$ with $v={\sigma}^{\frac{1}{p}}\lambda^{-\frac{1}{p}}$, whose method relies on a fine decomposition of the Calder\'{o}n--Zygmund kernel and uses the idea of the approximation of compact operators. We note that a combination of Theorems \ref{t-hlw} and \ref{t-u}, together with Theorems \ref{t-ck} and \ref{t-p}, readily gives another version of the two-weight compactness of commutators of Calder\'{o}n--Zygmund operators.
\begin{cor}\label{c-ll}
Let $T$ be a Calder\'{o}n--Zygmund operator, $p\in (1,\infty)$, two weights $\sigma, \lambda \in A_p$. Suppose $$b\in \overline{\bigcup_{\eta\in (1,r_{\sigma,\lambda}]}{\rm BMO}_{v^{\eta}}(\mathbb{R}^{d})\cap {\rm CMO}(\mathbb{R}^{d})}^{{\rm BMO}_{v}(\mathbb{R}^{d})},$$ where
\begin{align*}
v={\sigma}^{\frac{1}{p}}\lambda^{-\frac{1}{p}}\in A_2,\ \ \ \ r_{\sigma,\lambda}&=\min(\tilde{r}_{p,\sigma},\tilde{r}_{p,\lambda})\\
&=1+\frac{1}{2^{2\max(p,p')+d+1}\max\big([\sigma]^{\max(1,p'-1)}_{A_p},[\lambda]^{\max(1,p'-1)}_{A_p}\big)}.
\end{align*}
Then the commutator $[b, T]$ is compact from $L^{p}(\sigma)$ to $L^{p}(\lambda)$.
\end{cor}

Before giving our proof, we discuss the relationship between ${\rm BMO}_{v^{\eta}}(\mathbb{R}^{d})\cap {\rm CMO}(\mathbb{R}^{d})$ and ${\rm BMO}_{v}(\mathbb{R}^{d})$. Indeed, inspired by the work of Lerner, Ombrosi and Rivera-R\'{\i}os \cite[Lemma 4.8]{LOR}, we have the following result.

\begin{lem}\label{l-lor}
Let $v\in A_2$ and $r_v=1+\frac{1}{2^{5+d}[v]_{A_2}}$. Then for any $\eta\in (1,r_v]$, $${\rm BMO}_{v^{\eta}}(\mathbb{R}^{d})\cap {\rm CMO}(\mathbb{R}^{d})\subset {\rm CMO}_v(\mathbb{R}^d).$$
\end{lem}

\begin{proof}
From Theorem \ref{t-p} and the H\"{o}lder inequality, it readily follows that when $\eta\in (1,r_v]$, for any cube $Q\subset \mathbb{R}^d$,
$$\frac{1}{2}\laz v^{\eta}\raz_Q^{\frac{1}{\eta}}\leq \laz v\raz_Q\leq \laz v^{\eta}\raz_Q^{\frac{1}{\eta}},$$
which further implies that
\begin{align*}
\frac{1}{v(Q)} \int_Q |b(x)-\laz b\raz_Q|dx &\approx \frac{1}{v^{\eta}(Q)^{\frac{1}{\eta}} |Q|^{\frac{1}{\eta'}}}\int_Q |b(x)-\laz b\raz_Q|dx\\
&\approx \bigg(\frac{1}{v^{\eta}(Q)}\int_Q |b(x)-\laz b\raz_Q|dx\bigg)^{\frac{1}{\eta}} \bigg(\frac{1}{|Q|}\int_Q |b(x)-\laz b\raz_Q|dx\bigg)^{\frac{1}{\eta'}}\\
&\lesssim \|b\|^{\frac{1}{\eta}}_{{\rm BMO}_{v^{\eta}}(\mathbb{R}^d)} \bigg(\frac{1}{|Q|}\int_Q |b(x)-\laz b\raz_Q|dx\bigg)^{\frac{1}{\eta'}}.
\end{align*}
Then from this and the definition of ${\rm CMO}_v(\mathbb{R}^{d})$, we conclude that
\begin{align*}
{\rm BMO}_{v^{\eta}}(\mathbb{R}^{d})\cap {\rm CMO}(\mathbb{R}^{d}) \subset {\rm CMO}_v(\mathbb{R}^{d}).
\end{align*}
This completes the proof.
\end{proof}
Similarly, we conclude that if $v\in A_{\infty}$, let $r_v=1+\frac{1}{2^{d+11}[w]_{A_{\infty}}}$, then for any $\eta\in (1,r_v)$, $v^{\eta}\in A_{\infty}$ and
$${\rm BMO}_{v^{\eta}}(\mathbb{R}^{d})\cap {\rm CMO}(\mathbb{R}^{d})\subset {\rm CMO}_v(\mathbb{R}^d).$$

\begin{proof}[Proof of Corollary \ref{c-ll}]
First, by Theorem \ref{t-hlw}, we have that $$\|[b_0,T]-[b_1,T]\|_{L^p(\sigma)\to L^p(\lambda)}=\|[b_0-b_1,T]\|_{L^p(\sigma)\to L^p(\lambda)}\lesssim \|b_0-b_1\|_{{\rm BMO}_v(\mathbb{R}^d)},$$ provided that $b_0,b_1\in {\rm BMO}_v(\mathbb{R}^{d})$. Therefore it suffices to show the two-weight compactness of the commutator $[b,T]$ with $b\in {\rm BMO}_{v^{\eta}}(\mathbb{R}^{d})\cap {\rm CMO}(\mathbb{R}^{d})$ for $\eta\in (1,r_{\sigma,\lambda}]$.

From $b\in {\rm CMO}(\mathbb{R}^{d})$ and Theorem \ref{t-u}, the commutator $[b, T]$ is compact on $L^p(\mathbb{R}^d)$. Moreover, for $\sigma, \lambda \in A_p$, by Theorem \ref{t-p}, it follows that $\sigma^{\eta},\lambda^{\eta}\in A_p$ for $\eta\in (1,r_{\sigma,\lambda}]$. Then from $b\in {\rm BMO}_{v^\eta}(\mathbb{R}^d)$ and Theorem \ref{t-hlw}, $[b, T]$ is bounded from $L^p(\sigma^{\eta})$ to $L^p(\lambda^{\eta})$. Thus, from Theorem \ref{t-ck}, we obtain that $[b, T]$ is compact from $L^{p}(\sigma)$ to $L^{p}(\lambda)$ by taking $X_0=L^p(\sigma^{\eta})$, $Y_0=L^p(\lambda^{\eta})$, $X_1=Y_1=L^p(\mathbb{R}^d)$ and $\theta=1-1/\eta$. This completes the proof.
\end{proof}

\begin{remark}\label{r-ll}
Since our approach relies on the interpolation of compactness, the assumption of $b$ in Corollary \ref{c-ll} seems to be slightly stronger than $b\in {\rm CMO}_v(\mathbb{R}^{d})$ in \cite[Theorem 1.2]{LL} by Lemma $\ref{l-lor}$. However, when $d=1$, for any $v\in A_2$ and $\eta\in (1,r_v]$, we claim that
\begin{equation}\label{eq-lor}
\overline{{\rm BMO}_{v^{\eta}}(\mathbb{R})\cap {\rm CMO}(\mathbb{R})}^{{\rm BMO}_{v}(\mathbb{R})}={\rm CMO}_v(\mathbb{R}).
\end{equation}
To see this, note that by Theorem \ref{t-p}, we have $v^{\eta}\in A_2$. Then from \cite[Theorem 4.1]{LL}, we obtain that $$C^{\infty}_c(\mathbb{R})\subset {\rm BMO}_{v^{\eta}}(\mathbb{R}),\ \ \overline{C^{\infty}_c(\mathbb{R})}^{{\rm BMO}_{v}(\mathbb{R})}={\rm CMO}_v(\mathbb{R}),$$
from which \eqref{eq-lor} follows. But it is still unknown whether \eqref{eq-lor} holds when $d>1$, so recovering \cite[Theorem 1.2]{LL} by our approach might deserve further study.
\end{remark} \vspace{0.3cm}

\noindent {\bf 4.B. Commutators of fractional integral operators.} In our second application, we study the commutators of fractional integral operators. Recall the classical fractional integral operator (or Riesz potential): given $\alpha\in (0,d)$, for a Schwartz function $f\in \mathcal{S}(\mathbb{R}^d)$, define the fractional integral operator $I_{\alpha}$ by $$I_{\alpha}(f)(x):=\int_{\mathbb{R}^d}\frac{f(y)}{|x-y|^{d-\alpha}}dy.$$

The two-weight norm inequalities of Bloom type for commutators of fractional integrals have been considered in many literatures. For instance, Ding and Lu \cite[Theorem 1]{DL} established the two-weight boundedness of commutators of fractional integrals $T_{\Omega,\alpha}$ with rough kernels $\Omega\in L^{\infty}(S^{d-1})$. However, to obtain our two-weight compactness result, we need a refined upper bounded estimate established recently by Holmes, Rahm and Spencer \cite[Theorem 1.1]{HRS} as follows:
\begin{thm}[\cite{HRS}]\label{t-hrs}
Let $\alpha\in (0,d)$, $p,q\in (1,\infty)$ with $\frac{1}{p}=\frac{1}{q}+\frac{\alpha}{d}$ and $\sigma,\lambda\in A_{p,q}$. Suppose $b \in {\rm BMO}_{v}(\mathbb{R}^{d})$ with $v=\sigma\lambda^{-1}$. Then $$\|[b,I_{\alpha}]\|_{L^p(\sigma^p)\to L^q(\lambda^q)}\approx \|b\|_{{\rm BMO}_v(\mathbb{R}^d)}.$$
\end{thm}

For the application of Theorem \ref{t-ck}, we need the following unweighted compactness result of Wang \cite{W}.
\begin{thm}[\cite{W}]\label{t-w}
Let $\alpha\in (0,d)$, $p,q\in (1,\infty)$ with $\frac{1}{p}=\frac{1}{q}+\frac{\alpha}{d}$. Suppose $b\in {\rm CMO}(\mathbb{R}^d)$, then $[b,I_{\alpha}]:L^p(\mathbb{R}^d)\to L^q(\mathbb{R}^d)$ is a compact operator.
\end{thm}

A combination of Theorems \ref{t-hrs} and \ref{t-w}, together with Theorems \ref{t-ck} and \ref{t-p}, gives the following two-weight compactness of commutators of fractional integral operators which is a new result to our best knowledge.
\begin{cor}\label{c-new}
Let $\alpha\in (0,d)$, $p,q\in (1,\infty)$ with $\frac{1}{p}=\frac{1}{q}+\frac{\alpha}{d}$ and $\sigma,\lambda\in A_{p,q}$. Suppose $$b\in \overline{\bigcup_{\eta\in (1,r_{\sigma,\lambda}]}{\rm BMO}_{v^{\eta}}(\mathbb{R}^{d})\cap {\rm CMO}(\mathbb{R}^{d})}^{{\rm BMO}_{v}(\mathbb{R}^{d})},$$ where
\begin{align*}
v=\sigma\lambda^{-1}\in A_2,\ \ \ \ r_{\sigma,\lambda}&=\min(\tilde{r}_{1+p'/q,\sigma^{-p'}},\tilde{r}_{1+q/p',\lambda^q})\\
&=1+\frac{1}{2^{2\max(p'/q,q/p')+d+3}\max\big([\sigma]^{\max(p',q)}_{A_{p,q}},[\lambda]^{\max(p',q)}_{A_{p,q}}\big)}.
\end{align*}
Then the commutator $[b, I_{\alpha}]$ is compact from $L^{p}(\sigma^p)$ to $L^{q}(\lambda^q)$.
\end{cor}

\begin{proof}
First, by Theorem \ref{t-hrs}, we have that $$\|[b_0,I_{\alpha}]-[b_1,I_{\alpha}]\|_{L^p(\sigma^p)\to L^q(\lambda^q)}=\|[b_0-b_1,I_{\alpha}]\|_{L^p(\sigma^p)\to L^q(\lambda^q)}\approx \|b_0-b_1\|_{{\rm BMO}_v(\mathbb{R}^d)},$$ provided that $b_0,b_1\in {\rm BMO}_v(\mathbb{R}^{d})$. Therefore it suffices to show the two-weight compactness of the commutator $[b,I_{\alpha}]$ with $b\in {\rm BMO}_{v^{\eta}}(\mathbb{R}^{d})\cap {\rm CMO}(\mathbb{R}^{d})$ for $\eta\in (1,r_{\sigma,\lambda}]$.

From $b \in {\rm CMO}(\mathbb{R}^d)$ and Theorem \ref{t-w}, we deduce that the commutator $[b,I_{\alpha}]:L^p(\mathbb{R}^d)\to L^q(\mathbb{R}^d)$ is a compact operator. Observe that for any $w\in A_{p,q}$, $$[w]_{A_{p,q}}=[w^q]^{\frac{1}{q}}_{A_{1+q/p'}}=[w^{-p'}]^{\frac{1}{p'}}_{A_{1+p'/q}}.$$
It follows that for $\sigma,\lambda\in A_{p,q}$, we have $\sigma^{-p'}\in A_{1+p'/q}$ and $\lambda^q\in A_{1+q/p'}$. Then by Theorem \ref{t-p}, when $\eta\in (1,r_{\sigma,\lambda}]$, we see that $\sigma^{-p'\eta}\in A_{1+p'/q}$ and $\lambda^{q\eta}\in A_{1+q/p'}$. From this, we obtain $\sigma^{\eta},\lambda^{\eta}\in A_{p,q}$. Moreover, from $b \in {\rm BMO}_{v^{\eta}}(\mathbb{R}^{d})$ and Theorem \ref{t-hrs}, $[b,I_{\alpha}]$ is bounded from $L^p(\sigma^{p\eta})$ to $L^q(\lambda^{q\eta})$. Thus, from Theorem \ref{t-ck}, we obtain that $[b, I_{\alpha}]$ is compact from $L^{p}(\sigma^p)$ to $L^{q}(\lambda^q)$ by taking $X_0=L^p(\sigma^{p\eta})$, $Y_0=L^q(\lambda^{q\eta})$, $X_1=L^p(\mathbb{R}^d)$, $Y_1=L^q(\mathbb{R}^d)$ and $\theta=1-1/\eta$. This completes the proof.
\end{proof}

\section{Applications to the two-weight compactness of bilinear operators}\label{s-5}
In this section, we show that our approach applied to linear operators still works in the multilinear case. In the one-weight setting, this topic has been discussed carefully in \cite{HLb,COY,WX}. For simplicity, we only study the two-weight compactness of bilinear operators. To be precise, we consider the first order bilinear commutators of bilinear Calder\'{o}n--Zygmund operators defined as follows.

Let $\Delta :=\{(x,y_0,y_1)\in (\mathbb{R}^d)^3: x=y_0=y_1\}$. We say that $K: (\mathbb{R}^d)^3\backslash \Delta\to \mathbb{C}$ is a bilinear Calder\'{o}n--Zygmund kernel if $$|K(x,y_0,y_1)|\lesssim \frac{1}{(|x-y_0|+|x-y_1|)^{2d}},$$
and
\begin{align*}
&|K(x+h,y_0,y_1)-K(x,y_0,y_1)|+|K(x,y_0+h,y_1)-K(x,y_0,y_1)|\\
&+|K(x,y_0,y_1+h)-K(x,y_0,y_1)|\lesssim \frac{|h|^{\delta}}{(|x-y_0|+|x-y_1|)^{2d+\delta}},
\end{align*}
for all $\max(|x-y_0|,|x-y_1|)\geq 2|h|$ and some fixed $\delta\in (0,1]$. Let $T$ be a bilinear operator initially defined from $(\mathcal{S}(\mathbb{R}^d))^2$ to $\mathcal{S}'(\mathbb{R}^d)$. Then, we say that $T$ is a bilinear Calder\'{o}n--Zygmund operator if it extends to be bounded from $L^{p_0}(\mathbb{R}^d)\times L^{p_1}(\mathbb{R}^d)$ to $L^{p}(\mathbb{R}^d)$ for some $p_0,p_1\in (1,\infty)$, where $\frac{1}{p}=\frac{1}{p_0}+\frac{1}{p_1}$, and there exists a bilinear Calder\'{o}n--Zygmund kernel $K$ such that for all $f_0,f_1 \in C^{\infty}_c(\mathbb{R}^d)$, $$T(f_0,f_1)(x):=\int_{(\mathbb{R}^d)^2}K(x,y_0,y_1)f_0(y_0)f_1(y_1)dy_0dy_1,\ \ x\notin {\rm supp}(f_0) \cap {\rm supp}(f_1).$$  Given a bilinear Calder\'{o}n--Zygmund operator $T$, a pair of locally integrable functions ${\bf b}=(b_0,b_1)$ and $\mathcal{I}(\not= \emptyset)\subset \{0,1\}$, we define the first order bilinear commutators as follows:
\begin{equation}\label{eq-li}
[{\bf b},T]_{\mathcal{I}}(f_0,f_1)=
\begin{cases}
b_0T(f_0,f_1)-T(b_0f_0,f_1),& \mathcal{I}=\{0\};\\
b_1T(f_0,f_1)-T(f_0,b_1f_1),& \mathcal{I}=\{1\};\\
[{\bf b},[{\bf b},T]_{\{0\}}]_{\{1\}}(f_0,f_1)=[{\bf b},[{\bf b},T]_{\{1\}}]_{\{0\}}(f_0,f_1),& \mathcal{I}=\{0,1\}.
\end{cases}	
\end{equation}

In the bilinear case, we need a bilinear version of the interpolation theorem of compactness due to Cobos, Fern\'{a}ndez-Cabrera and Mart\'{\i}nez \cite[Theorem 3.2]{CFM}. Before this, we need to introduce some notations. For a Banach couple $\bar{X}=(X_0,X_1)$, we write $X^{\circ}_i$ for the closure of $X_0\cap X_1$ in the norm of $X_i$ for $i=0,1$. We denote by $\mathcal{B}(\bar{X}\times \bar{Y}, \bar{Z})$ the space of all bilinear operators $T$ which are defined from $(X_0\cap X_1)\times (Y_0\cap Y_1)$ to $Z_0\cap Z_1$ and satisfy $$\|T(f_0,f_1)\|_{Z_i}\lesssim \|f_0\|_{X_i}\|f_1\|_{Y_i},\ \ f_0\in X_0\cap X_1,\ f_1\in Y_0\cap Y_1,\ i=0,1,$$ where $\bar{Y}:=(Y_0,Y_1)$ and $\bar{Z}:=(Z_0,Z_1)$ are Banach couples.
\begin{thm}[\cite{CFM}]\label{t-cfm}
Let $\bar{X}=(X_0,X_1)$ and $\bar{Y}=(Y_0,Y_1)$ be Banach couples. Assume that $(\Omega, \mu)$ is a $\sigma$-finite measure space. Let $\bar{Z}=(Z_0,Z_1)$ be a couple of Banach function spaces on $\Omega$, $\theta\in (0,1)$ and $T\in \mathcal{B}(\bar{X}\times \bar{Y}, \bar{Z})$. If $T: X^{\circ}_0 \times Y^{\circ}_0 \to Z_0$ compactly and $Z_0$ has absolutely continuous norm, then $T$ can be uniquely extended to a compact bilinear operator from $[X_0,X_1]_{\theta}\times [Y_0,Y_1]_{\theta}$ to $[Z_0,Z_1]_{\theta}$.
\end{thm}
For the definitions of the Banach function space and the absolutely continuous norm, see, for example, \cite{CFM,HLb}. It is known that the assumptions in Theorem \ref{t-cfm} hold for the weighted Lebesgue spaces, see \cite[Lemma 3.3]{HLb}, which is sufficient for our applications of Theorem \ref{t-cfm}.

We recall that a weight vector ${\bf w}=(w_0,w_1)$ is called a (multiple) $A_{\bf p}$ weight (write ${\bf w}\in A_{\bf p}$) if $$[{\bf w}]_{A_{\bf p}}:=\sup_{Q\subset \mathbb{R}^d}\laz w_{\bf p}\raz^{\frac{1}{p}}_Q \laz w^{1-p'_0}_0\raz^{\frac{1}{p'_0}}_Q \laz w^{1-p'_1}_1\raz^{\frac{1}{p'_1}}_Q< \infty,$$
where ${\bf p}=(p_0,p_1)\in (1,\infty)^2$, $\frac{1}{p}=\frac{1}{p_0}+\frac{1}{p_1}$
and $w_{\bf p}=w^{\frac{p}{p_0}}_0 w^{\frac{p}{p_1}}_1$. The following is a characterization of $A_{\bf p}$ weights which is quite useful in applications:
\begin{equation}\label{eq-m}
{\bf w}\in A_{\bf p} \iff w_{\bf p}\in A_{2p},\ w^{1-p'_i}_i\in A_{2p'_i},\ i=0,1.
\end{equation}

The two-weight boundedness of Bloom type of multilinear commutators of multilinear Calder\'{o}n--Zygmund operators were first studied by Kunwar and Ou in \cite{KO}. Their result was then improved by Li \cite{L} in terms of the genuinely multilinear weights. The following is a bilinear version of \cite[Theorem 3.5]{L}.
\begin{thm}[\cite{L}]\label{t-li}
Let $T$ be a bilinear Calder\'{o}n--Zygmund operator and $\mathcal{I}(\not= \emptyset)\subset \{0,1\}$. Let $[{\bf b}, T]_{\mathcal{I}}$ be defined as in \eqref{eq-li}, $p_0,p_1\in (1,\infty)$ and $\frac{1}{p}=\frac{1}{p_0}+\frac{1}{p_1}$. Suppose that $(w_0,w_1)\in A_{\bf p}$, where
$$w_i=\begin{cases}
\sigma_i,& i\notin \mathcal{I};\\
\sigma_i\ {\rm or}\ \lambda_i,& i\in \mathcal{I}.
\end{cases}$$
If $v_i:=\sigma^{\frac{1}{p_i}}_i\lambda^{-\frac{1}{p_i}}_i\in A_{\infty}$ and $b_i\in {\rm BMO}_{v_i}(\mathbb{R}^d)$ for each $i\in \mathcal{I}$, then
$$\|[{\bf b}, T]_{\mathcal{I}}\|_{L^{p_0}(\sigma_0)\times L^{p_1}(\sigma_1)\to L^p(\prod_{i\in \mathcal{I}}\lambda^{\frac{p}{p_i}}_i \prod_{j\notin \mathcal{I}}\sigma^{\frac{p}{p_j}}_j)}\lesssim \prod_{i\in \mathcal{I}}\|b_i\|_{{\rm BMO}_{v_i}(\mathbb{R}^d)}.$$
\end{thm}

We would like to mention that different from the linear case, in the multilinear case, by \eqref{eq-m}, we only have $\sigma^{1-p'_i}_i,\lambda^{1-p'_i}_i\in A_{2p'_i}$ for each $i\in \mathcal{I}$, which implies $v^{\frac{1}{2}}_i\in A_{2}$, but may not ensure $v_i\in A_{\infty}$, see \cite[Example 2.12]{L} for more details.

For the application of Theorem \ref{t-cfm}, we need the following unweighted compactness result of B\'{e}nyi and Torres \cite[Theorem 1]{BT}, see also \cite{TXYY} for some developments in the one-weight setting.
\begin{thm}[\cite{BT}]\label{t-bt}
Let $T$ be a bilinear Calder\'{o}n--Zygmund operator and $\mathcal{I}(\not= \emptyset)\subset \{0,1\}$. Let $[{\bf b}, T]_{\mathcal{I}}$ be defined as in \eqref{eq-li}, $p_0,p_1,p\in (1,\infty)$ such that $\frac{1}{p}=\frac{1}{p_0}+\frac{1}{p_1}$. If $b_i \in {\rm CMO}(\mathbb{R}^d)$ for each $i\in \mathcal{I}$, then $[{\bf b}, T]_{\mathcal{I}}$ is compact from $L^{p_0}(\mathbb{R}^d)\times L^{p_1}(\mathbb{R}^d)$ to $L^{p}(\mathbb{R}^d)$.
\end{thm}

A combination of Theorems \ref{t-li} and \ref{t-bt}, together with Theorems \ref{t-cfm} and \ref{t-p}, gives the following two-weight compactness of bilinear commutators of bilinear Calder\'{o}n--Zygmund operators which is a new result to our best knowledge.
\begin{cor}\label{c-twb}
Let $T$ be a bilinear Calder\'{o}n--Zygmund operator and $\mathcal{I}(\not= \emptyset)\subset \{0,1\}$. Let $[{\bf b}, T]_{\mathcal{I}}$ be defined as in \eqref{eq-li}, $p_0,p_1,p\in (1,\infty)$ such that $\frac{1}{p}=\frac{1}{p_0}+\frac{1}{p_1}$. Suppose that $(w_0,w_1)\in A_{\bf p}$, where
$$w_i=\begin{cases}
\sigma_i,& i\notin \mathcal{I};\\
\sigma_i\ {\rm or}\ \lambda_i,& i\in \mathcal{I}.
\end{cases}$$
If $v_i:=\sigma^{\frac{1}{p_i}}_i\lambda^{-\frac{1}{p_i}}_i\in A_{\infty}$ and $$b_i\in \overline{\bigcup_{\eta_{\mathcal{I}}\in (1,r_{\sigma_{\mathcal{I}},\lambda_{\mathcal{I}}})}{\rm BMO}_{v_i^{\eta_{\mathcal{I}}}}(\mathbb{R}^{d})\cap {\rm CMO}(\mathbb{R}^{d})}^{{\rm BMO}_{v_i}(\mathbb{R}^{d})}$$ for each $i\in \mathcal{I}$, where $$r_{\sigma_{\mathcal{I}},\lambda_{\mathcal{I}}}:=\min\bigg(\tilde{r}_{2p,\sigma^{\frac{p}{p_0}}_0\sigma^{\frac{p}{p_1}}_1},\tilde{r}_{2p,\prod_{i\in \mathcal{I}}\lambda_i^{\frac{p}{p_i}} \prod_{j\notin \mathcal{I}}\sigma_j^{\frac{p}{p_j}}}, \tilde{r}_{2p'_0,\sigma_0^{1-p'_0}}, \tilde{r}_{2p'_1,\sigma_1^{1-p'_1}}, \min_{i\in \mathcal{I}}\tilde{r}_{2p'_i,\lambda_i^{1-p'_i}}, \min_{i\in\mathcal{I}}r_{v_i}\bigg).$$ Then $[{\bf b}, T]_{\mathcal{I}}$ is compact from $L^{p_0}(\sigma_0)\times L^{p_1}(\sigma_1)$ to $L^p(\prod_{i\in \mathcal{I}}\lambda^{\frac{p}{p_i}}_i \prod_{j\notin \mathcal{I}}\sigma^{\frac{p}{p_j}}_j)$.
\end{cor}

\begin{proof}
We only give the detailed proof of the corollary in the case $\mathcal{I}=\{0,1\}$, since the proofs of the other two cases are similar and easier. First, by Theorem \ref{t-li}, we have that
\begin{align*}
&\|[{\bf b},T]_{\mathcal{I}}-[{\bf b^*},T]_{\mathcal{I}}\|\\
&=\|[{\bf b},[{\bf b},T]_{\{0\}}]_{\{1\}}-[{\bf b^*},[{\bf b},T]_{\{0\}}]_{\{1\}}+[{\bf b^*},[{\bf b},T]_{\{0\}}]_{\{1\}}-[{\bf b^*},[{\bf b^*},T]_{\{0\}}]_{\{1\}}\|\\
&\leq \|[{\bf b-b^*},[{\bf b},T]_{\{0\}}]_{\{1\}}\| + \|[{\bf b^*},[{\bf b-b^*},T]_{\{0\}}]_{\{1\}}\|\\
&\lesssim \|b_0\|_{{\rm BMO}_{v_0}(\mathbb{R}^d)} \|b_1-b^*_1\|_{{\rm BMO}_{v_1}(\mathbb{R}^d)} + \|b^*_1\|_{{\rm BMO}_{v_1}(\mathbb{R}^d)} \|b_0-b^*_0\|_{{\rm BMO}_{v_0}(\mathbb{R}^d)},
\end{align*}
provided that ${\bf b}=(b_0,b_1), {\bf b^*}=(b^*_0,b^*_1)\in {\rm BMO}_{v_0}(\mathbb{R}^d)\times  {\rm BMO}_{v_1}(\mathbb{R}^d)$, where $$\|\cdot\|:=\|\cdot\|_{L^{p_0}(\sigma_0)\times L^{p_1}(\sigma_1)\to L^p(\lambda^{\frac{p}{p_0}}_0\lambda^{\frac{p}{p_1}}_1)}.$$
Therefore, it suffices to show the two-weight compactness of the bilinear commutator $[{\bf b},T]_{\mathcal{I}}$ with $b_i\in {\rm BMO}_{v_i^{\eta_{\mathcal{I}}}}(\mathbb{R}^{d})\cap {\rm CMO}(\mathbb{R}^{d})$ for each $i\in \mathcal{I}$, where $\eta_{\mathcal{I}}\in (1,r_{\sigma_{\mathcal{I}},\lambda_{\mathcal{I}}})$.

From $b_i\in {\rm CMO}(\mathbb{R}^{d})$ for each $i\in \mathcal{I}$ and Theorem \ref{t-bt}, we deduce that $[{\bf b}, T]_{\mathcal{I}}$ is compact from $L^{p_0}(\mathbb{R}^d)\times L^{p_1}(\mathbb{R}^d)$ to $L^{p}(\mathbb{R}^d)$. Next, for $(\sigma_0,\sigma_1),(\lambda_0,\lambda_1)\in A_{\bf p}$, from \eqref{eq-m},
$$\sigma^{\frac{p}{p_0}}_0\sigma^{\frac{p}{p_1}}_1, \lambda^{\frac{p}{p_0}}_0\lambda^{\frac{p}{p_1}}_1\in A_{2p},\ \sigma^{1-p'_i}_i,\lambda^{1-p'_i}_i\in A_{2p'_i},\ i=0,1.$$ Then by Theorem \ref{t-p}, when $\eta_{\mathcal{I}}\in (1,r_{\sigma_{\mathcal{I}},\lambda_{\mathcal{I}}})$,  we see that
$$\sigma^{\frac{\eta_{\mathcal{I}} p}{p_0}}_0\sigma^{\frac{\eta_{\mathcal{I}} p}{p_1}}_1, \lambda^{\frac{\eta_{\mathcal{I}} p}{p_0}}_0\lambda^{\frac{\eta_{\mathcal{I}} p}{p_1}}_1\in A_{2p},\ \sigma^{\eta_{\mathcal{I}}(1-p'_i)}_i,\lambda^{\eta_{\mathcal{I}}(1-p'_i)}_i\in A_{2p'_i},\ v^{\eta_{\mathcal{I}}}_i\in A_{\infty},\ i=0,1.$$ From this, again by \eqref{eq-m}, we have $(\sigma^{\eta_{\mathcal{I}}}_0,\sigma^{\eta_{\mathcal{I}}}_1),(\lambda^{\eta_{\mathcal{I}}}_0,\lambda^{\eta_{\mathcal{I}}}_1)\in A_{\bf p}$. Note that from $b_i\in {\rm BMO}_{v^{\eta_{\mathcal{I}}}_i}(\mathbb{R}^d)$ for each $i\in \mathcal{I}$ and Theorem \ref{t-li}, we deduce that $[{\bf b},T]_{\mathcal{I}}$ is bounded from $L^{p_0}(\sigma^{\eta_{\mathcal{I}}}_0)\times L^{p_1}(\sigma^{\eta_{\mathcal{I}}}_1)$ to $L^p(\lambda^{\frac{\eta_{\mathcal{I}} p}{p_0}}_0\lambda^{\frac{\eta_{\mathcal{I}} p}{p_1}}_1)$. Thus, by Theorem \ref{t-cfm} with
\begin{align*}
(X_0,X_1)&=(L^{p_0}(\mathbb{R}^d),L^{p_0}(\sigma^{\eta_{\mathcal{I}}}_0)),\ \ (Y_0,Y_1)=(L^{p_1}(\mathbb{R}^d),L^{p_1}(\sigma^{\eta_{\mathcal{I}}}_1)),\\
(Z_0,Z_1)&=(L^{p}(\mathbb{R}^d),L^p(\lambda^{\frac{\eta_{\mathcal{I}} p}{p_0}}_0\lambda^{\frac{\eta_{\mathcal{I}} p}{p_1}}_1)),\ \ \theta=\frac{1}{\eta_{\mathcal{I}}},
\end{align*}
we obtain that $[{\bf b},T]_{\mathcal{I}}$ is compact from $L^{p_0}(\sigma_0)\times L^{p_1}(\sigma_1)$ to $L^p(\lambda^{\frac{p}{p_0}}_0\lambda^{\frac{p}{p_1}}_1)$. This completes the proof.
\end{proof}

\begin{remark}
From our applications in Sections \ref{s-4} and \ref{s-5}, we see that the two-weight compactness on weighted Lebesgue spaces of certain (multi-)linear operators is naturally obtained from the unweighted compactness and the two-weight boundedness without dealing with any concrete forms and characteristics of these operators. Based on this fact, our results may still valid for more general (multi-)linear integral operators, e.g., $\omega$-Calder\'{o}n--Zygmund operators with $\omega$ satisfying the Dini type estimates studied in \cite{LOR}.
\end{remark}

\bigskip
\noindent {\bf Acknowledgments:} Wu is supported by the NNSF of China (Grant Nos. 11871101 and 11771358). Yang is supported by the NNSF of China (Grant Nos. 11971402 and 11871254).


\begin{thebibliography}{10}

\bibitem{BT}
A.~B\'{e}nyi and R.~H. Torres.
\newblock Compact bilinear operators and commutators.
\newblock {\em Proc. Amer. Math. Soc.}, 141(10):3609--3621, 2013.

\bibitem{BL}
J.~Bergh and J.~L\"{o}fstr\"{o}m.
\newblock {\em Interpolation spaces. {A}n introduction}.
\newblock Springer-Verlag, Berlin-New York, 1976.
\newblock Grundlehren der Mathematischen Wissenschaften, No. 223.

\bibitem{B}
S.~Bloom.
\newblock A commutator theorem and weighted {BMO}.
\newblock {\em Trans. Amer. Math. Soc.}, 292(1):103--122, 1985.

\bibitem{COY}
M.~Cao, A.~Olivo, and K.~Yabuta.
\newblock Extrapolation for multilinear compact operators and applications.
\newblock {\em arXiv:2011.13191}, 2020.

\bibitem{CC}
A.~Clop and V.~Cruz.
\newblock Weighted estimates for {B}eltrami equations.
\newblock {\em Ann. Acad. Sci. Fenn. Math.}, 38(1):91--113, 2013.

\bibitem{CFM}
F.~Cobos, L.~M. Fern\'{a}ndez-Cabrera, and A.~Mart\'{\i}nez.
\newblock On compactness results of {L}ions-{P}eetre type for bilinear
  operators.
\newblock {\em Nonlinear Anal.}, 199:111951, 9, 2020.

\bibitem{C}
M.~Cwikel.
\newblock Real and complex interpolation and extrapolation of compact
  operators.
\newblock {\em Duke Math. J.}, 65(2):333--343, 1992.

\bibitem{CK}
M.~Cwikel and N.~J. Kalton.
\newblock Interpolation of compact operators by the methods of {C}alder\'{o}n
  and {G}ustavsson-{P}eetre.
\newblock {\em Proc. Edinburgh Math. Soc. (2)}, 38(2):261--276, 1995.

\bibitem{DL}
Y.~Ding and S.~Lu.
\newblock Higher order commutators for a class of rough operators.
\newblock {\em Ark. Mat.}, 37(1):33--44, 1999.

\bibitem{HLW}
I.~Holmes, M.~T. Lacey, and B.~D. Wick.
\newblock Commutators in the two-weight setting.
\newblock {\em Math. Ann.}, 367(1-2):51--80, 2017.

\bibitem{HRS}
I.~Holmes, R.~Rahm, and S.~Spencer.
\newblock Commutators with fractional integral operators.
\newblock {\em Studia Math.}, 233(3):279--291, 2016.

\bibitem{H}
T.~Hyt\"{o}nen.
\newblock Extrapolation of compactness on weighted spaces.
\newblock {\em arXiv:2003.01606}, 2020.

\bibitem{HL}
T.~Hyt\"{o}nen and S.~Lappas.
\newblock Extrapolation of compactness on weighted spaces {I}{I}:
  {O}ff-diagonal and limited range estimates.
\newblock {\em arXiv:2006.15858}, 2020.

\bibitem{HLb}
T.~Hyt\"{o}nen and S.~Lappas.
\newblock Extrapolation of compactness on weighted spaces {I}{I}{I}: {B}ilinear
  operators.
\newblock {\em arXiv:2012.10407}, 2020.

\bibitem{KO}
I.~Kunwar and Y.~Ou.
\newblock Two-weight inequalities for multilinear commutators.
\newblock {\em New York J. Math.}, 24:980--1003, 2018.

\bibitem{LL}
M.~Lacey and J.~Li.
\newblock Compactness of commutator of {R}iesz transforms in the two weight
  setting.
\newblock {\em arXiv:2010.15451}, 2020.

\bibitem{LOR}
A.~K. Lerner, S.~Ombrosi, and I.~P. Rivera-R\'{\i}os.
\newblock Commutators of singular integrals revisited.
\newblock {\em Bull. Lond. Math. Soc.}, 51(1):107--119, 2019.

\bibitem{L}
K.~Li.
\newblock Multilinear commutators in the two-weight setting.
\newblock {\em arXiv:2006.09071}, 2020.

\bibitem{P}
C.~P\'{e}rez.
\newblock A course on singular integrals and weights.
\newblock {\em Advanced courses in Mathematics at the CRM-Barcelona series,
  Birkha\"{u}ser}, 2013.

\bibitem{SW}
E.~M. Stein and G.~Weiss.
\newblock Interpolation of operators with change of measures.
\newblock {\em Trans. Amer. Math. Soc.}, 87(1):159--172, 1958.

\bibitem{SVW}
C.~B. Stockdale, P.~Villarroya, and B.~D. Wick.
\newblock Sparse domination results for compactness on weighted spaces.
\newblock {\em arXiv:1912.10290}, 2019.

\bibitem{TXYY}
J.~Tao, Q.~Xue, D.~Yang, and W.~Yuan.
\newblock {XMO} and weighted compact bilinear commutators.
\newblock {\em arXiv:1909.03173}, 2019.

\bibitem{U}
A.~Uchiyama.
\newblock On the compactness of operators of {H}ankel type.
\newblock {\em Tohoku Math. J. (2)}, 30(1):163--171, 1978.

\bibitem{W}
S.~Wang.
\newblock The compactness of the commutator of fractional integral operator.
\newblock {\em Chin. Ann. Math}, 8:475--482, 1987.

\bibitem{WX}
S.~Wang and Q.~Xue.
\newblock On weighted compactness of commutators of bilinear maximal
  {C}alder\'{o}n-{Z}ygmund singular integral operators.
\newblock {\em arXiv:2012.12747}, 2020.

\end{thebibliography}
\end{document}